\documentclass[11pt,reqno]{amsart}
\usepackage[T1]{fontenc}
\usepackage{lmodern}
\usepackage[margin=1.03in]{geometry}
\usepackage{amsmath,amssymb,amsthm,mathtools,booktabs,tabularx,array}
\usepackage[expansion=false]{microtype}
\usepackage[hidelinks]{hyperref}
\hypersetup{pdftitle={Rank-two Nahm sums, modularity obstructions, and the McKay--Thompson series T31A},pdfauthor={Cetin Hakimoglu-Brown}}
\newcommand{\QQ}{\mathbb Q}
\newcommand{\ZZ}{\mathbb Z}
\newcommand{\RR}{\mathbb R}
\newcommand{\HH}{\mathbb H}
\newcommand{\EE}{\mathbb E}
\newcommand{\Li}{\operatorname{Li}}
\newcommand{\diag}{\operatorname{diag}}
\newcommand{\Res}{\operatorname{Res}}
\newcommand{\RG}{\mathcal G}
\newcommand{\RH}{\mathcal H}
\newcommand{\Ssum}{\mathcal S_{31}}
\newtheorem{theorem}{Theorem}[section]
\newtheorem{proposition}[theorem]{Proposition}
\newtheorem{lemma}[theorem]{Lemma}
\newtheorem{corollary}[theorem]{Corollary}
\newtheorem{conjecture}[theorem]{Conjecture}
\theoremstyle{remark}

\numberwithin{equation}{section}
\title[Rank-two Nahm sums and the series $T_{31A}$]
{Rank-two Nahm sums, modularity obstructions,\break and the McKay--Thompson series $T_{31A}$}
\author{Cetin Hakimoglu-Brown}
\email{mathemails@proton.me}
\date{20 September 2026}
\begin{document}
\begin{abstract}
We conjecture a positive four-shift identity for rank-two Nahm sums
with denominator steps $(1,31)$, giving a mixed-base realization of
the Monster McKay--Thompson series $T_{31A}$ through the classical
Rogers--Ramanujan function $U(1,31)$. A theta dissection directly
derives the modular target. For the underlying quadratic exponent
$3r^2+31rs+93s^2$, we prove nonmodularity of every individual sum
with rational linear terms and normalization, together with an
exclusion from finite Fricke systems with meromorphic finite-width
expansions at infinity. Exact agreement through $q^{2000}$ and six
vanishing radial corrections support the proposed combination;
a five-term dilogarithm certificate and an exact amplitude calculation
prove its leading asymptotic behavior. We further propose a
three-component Fricke law for a dual combination, with an explicit
icosahedral matrix involving fifth-root constants, and a quadratic
identity linking its components to $T_{31A}$. The dual translation
law and matrix algebra are proved; its Fricke transformation and
moonshine identity remain conjectural, supported respectively by
numerical evaluations and exact finite coefficient checks.
\end{abstract}
\maketitle

\section{Introduction}
An individual Nahm sum can fail the necessary asymptotic conditions for
modularity even when a finite combination of its shifts is modular.
This distinction is already present in the literature
\cite{AVH,Storzer}. It becomes decisive for the last positive integral
quadratic datum in the restricted cubic sieve of \cite{Brown}.
That sieve leaves denominator indices $3$, $13$, and $31$ within a
specified ten-record monomial family. The index-$3$ sums give the
Kanade--Russell identities modulo nine; the index-$13$ evaluations in
\cite{Brown} remain conjectural. Here we investigate the index-$31$
system beyond the individual-sum obstruction.

Put $q=e^{2\pi i\tau}$, where
$\tau\in\HH=\{z\in\mathbb C:\operatorname{Im}z>0\}$, and
\[
 (a;q)_n=\prod_{j=0}^{n-1}(1-aq^j),\qquad
 (a_1,\ldots,a_k;q)_\infty=\prod_{i=1}^k\prod_{j\geq0}(1-a_iq^j).
\]
Fractional powers mean $q^c=e^{2\pi i c\tau}$. Define
\begin{equation}\label{eq:family}
 F_{a,b}(q)=\sum_{r,s\geq0}
 \frac{q^{3r^2+31rs+93s^2+ar+bs}}
 {(q;q)_r(q^{31};q^{31})_s},\qquad a,b\in\QQ.
\end{equation}
The associated matrices are
\begin{equation}\label{eq:matrices}
 A=\begin{pmatrix}6&1\\31&6\end{pmatrix},\quad
 D=\diag(1,31),\quad
 B=AD=\begin{pmatrix}6&31\\31&186\end{pmatrix},\quad \det B=155.
\end{equation}
Positive definiteness gives local uniform convergence on $\HH$ for all
fixed rational $a,b$. Our first result covers this entire parameter set.

\begin{theorem}\label{thm:individual}
For every $a,b,c\in\QQ$, the function $q^cF_{a,b}(q)$ is not a
weight-zero modular function on any finite-index subgroup of
$\mathrm{SL}_2(\ZZ)$ with finite-order multiplier, meromorphic at the cusps.
\end{theorem}

The proof in Section~\ref{sec:exclusion} uses the first two radial
corrections in the cyclic cubic saddle field of conductor $19$.
The first condition forces $(a,b)=(19/3,124/3)$; the second logarithmic
correction at this pair is the nonzero rational number
$-1653155/3518667$. This is an exact exclusion over rational parameters,
independent of a proposed product modulus or a bounded search.
Corollary~\ref{cor:fricke} strengthens this conclusion: for every $N>0$,
no finite-dimensional space preserved by $\tau\mapsto-1/(N\tau)$ and
meromorphic at infinity in finite-width Fourier parameters can contain
an individual $q^cF_{a,b}$. This statement does not assume a
finite-image modular representation.

For the combination problem, let
\begin{equation}\label{eq:RR}
 \RG(q)=\sum_{n\geq0}\frac{q^{n^2}}{(q;q)_n},\qquad
 \RH(q)=\sum_{n\geq0}\frac{q^{n^2+n}}{(q;q)_n},
\end{equation}
and use the classical notation \cite{BY}
\begin{equation}\label{eq:U31}
 U(1,31)=\RH(q)\RG(q^{31})-q^6\RG(q)\RH(q^{31}).
\end{equation}
Set
\begin{equation}\label{eq:signed}
 \Ssum(q)=F_{0,0}+q^2F_{3,31}-q^{22}F_{16,93}.
\end{equation}

\begin{conjecture}\label{conj:main}
For $|q|<1$,
\begin{equation}\label{eq:mainconjecture}
 \Ssum(q)=U(1,31).
\end{equation}
\end{conjecture}

Proposition~\ref{prop:positive} proves the positive reformulation
\begin{equation}\label{eq:positive}
 \Ssum(q)=F_{0,0}+q^2F_{5,31}+q^8F_{9,62}+q^9F_{11,62}.
\end{equation}
Thus the conjecture proposes a positive coupled expression for a known
function. Four constituents still use only two summation variables.
The classical definition \eqref{eq:U31} itself gives an uncoupled signed
double sum; the specific quadratic datum \eqref{eq:matrices} and its
positive combination are the point of the proposal.

The moonshine connection makes the target particularly concrete:
Section~\ref{sec:target} proves, independently of the conjecture, that
\begin{equation}\label{eq:moonshine-intro}
 T_{31A}(\tau)=q^{-1}U(1,31)^3.
\end{equation}
Consequently Conjecture~\ref{conj:main} is equivalent to
$q^{-1}\Ssum(q)^3=T_{31A}(\tau)$, with the cube-root branch fixed by
$\Ssum(q)=1+O(q^2)$. This proposes a coupled positive Nahm-sum
realization of a classical moonshine function. We give a direct
derivation through a discriminant-$-31$ theta difference and Fricke
symmetry. The target and its moonshine interpretation are classical
\cite{CN,OEIS31}; the proposed mixed-base identity is the new assertion.
The distinction between the proved target and its proposed Nahm-sum
realization is maintained throughout. The coefficient and saddle
calculations in Section~\ref{sec:evidence} are finite exact evidence,
not a proof of \eqref{eq:mainconjecture}.

Section~\ref{sec:dual-fricke} develops a second proposed moonshine
connection, using the dual quadratic datum $A^{-1}D$. An explicit
dual combination splits into three fractional exponent classes. We
conjecture that Fricke acts on these components by a constant matrix
involving $\cos(2\pi/5)$ and $\cos(4\pi/5)$, and that a quadratic
expression in the components equals $(T_{31A}+1)^2-4$.
The exact translation law and the icosahedral matrix relations are
proved independently of these conjectures. This additional structure
concerns the dual combination; it does not supply a proof of
Conjecture~\ref{conj:main}.

\subsection{Relation to earlier work}
Zagier's asymptotic classification in rank one illustrates how
successive radial coefficients constrain modularity; he also explains
the rapidly increasing difficulty in higher rank
\cite[Section~II.3.C]{Zagier}. Vlasenko--Zwegers develop the higher-rank
method \cite{VZ}, and Mizuno treats symmetrizable matrices \cite{Mizuno}.
The most direct mixed-base rank-two precedent is Wang--Wang
\cite{WW}: they prove Rogers--Ramanujan type evaluations for eight
further sets of Mizuno's candidates and a conjectural vector-valued
transformation law.
Here the arithmetic construction fixes the matrix first. The remaining
classification over all rational linear terms reduces to two polynomial
equations in the cubic saddle field, followed by one nonzero rational
correction. This is a complete exclusion for \eqref{eq:matrices}, not
a classification of all rank-two matrices.
The broader monomial construction, the ten-record source, and the
index-$13$ comparison are developed in \cite{Brown}; their proofs are
not repeated here. In particular, \cite[version~1, Section~5.3]{Brown}
leaves the unbounded index-$31$ individual-sum problem as a further
asymptotic test. Theorem~\ref{thm:individual} completes that test, while
Corollary~\ref{cor:fricke} states its consequence for Fricke closure.
Section~\ref{sec:arithmetic} adds a positive-support obstruction that
applies to any finite number of shifts. The new candidate and its exact
checks address the combination problem left open by the exclusion.

Known modular combinations of nonmodular sums include the Ising-model
identities of Andrews, van Ekeren and Heluani \cite{AVH} and the families
in Storzer's thesis \cite[Section~6.3]{Storzer}. We give an explicit
positive modulus-five reformulation below. These precedents show why
Theorem~\ref{thm:individual} does not settle the combination problem.
The classical family $U(r,s)$ and its relations with binary quadratic
forms are treated by Berkovich--Yesilyurt \cite{BY}. The theta form of
discriminant $-31$ used here is explicitly tabulated by Akbary--Totani
\cite{AT}. Accordingly, neither the target function nor the general
phenomenon of modular combinations is presented as new.

In the standard Andrews--Gordon representation, product modulus
$M=2k+1$ corresponds to $k-1$ summation variables \cite{Andrews}; see also
\cite[(1.2)]{Warnaar}.
This gives ranks $2,3,5,14$ at moduli $7,9,13,31$. Smaller representations
can concern different functions: the Kanade--Russell sums at modulus
nine have rank two, with denominator steps $(1,3)$, and their evaluations
are now proved \cite{MizunoKR,Xia}. Our rank-two conjecture at $31$
does not compress an Andrews--Gordon identity for the same function.
Rank counts summation variables; it does not specify the dimension of
a modular representation. In particular, the scalar Fricke law below
does not assert a two-component transformation under the full modular group.

\section{Known evaluations and obstructions in the cubic source}
\label{sec:arithmetic}
\subsection{Denominator index, field conductor, and product modulus}
These are distinct parameters. Write a symmetrizable quadratic datum as
$D=\diag(1,m)$, $B=AD=B^{\mathsf T}$, with exponent
$Q(r,s)=\tfrac12(r,s)B(r,s)^{\mathsf T}$.
The ten normalized records in \cite[Table~3]{Brown} come from
\[
 y=x^p(1-x)^h,\qquad 1-y=x^u(1-x)^v,\qquad m=uh-pv>0.
\]
When $h\ne0$, their conversion is
\begin{equation}\label{eq:conversion}
 A=\frac1h\begin{pmatrix}-p&1\\m&v\end{pmatrix},\qquad
 \det(AD)=-\frac{mu}{h}.
\end{equation}
Table~\ref{tab:source} records the outcomes needed here. It refers only
to that finite source and the displayed orientations. It is not a
classification of all two-term dilogarithms or all modular multisums.

\begin{table}[htbp]
\caption{The ten-record source of \cite{Brown}. Here $N$ is the saddle-field
conductor and $m$ the second denominator step. The last column distinguishes
matrix restrictions from evaluation status.}\label{tab:source}
\centering\small
\renewcommand{\arraystretch}{1.22}
\begin{tabularx}{\textwidth}{@{}rrrlX@{}}\toprule
Row&$N$&$m$&$Q(r,s)$&Outcome\\\midrule
1&7&2&$r^2/2+rs+s^2$&Known parity-component construction \cite{Mizuno}\\
2&7&2&---&No rational matrix in this orientation; reflected matrix singular\\
3&9&3&$r^2+3rs+3s^2$&Proved Kanade--Russell evaluations\\
4&7&7&$(r^2+7rs+7s^2)/3$&Indefinite; positive-support obstruction\\
5&13&13&$2r^2+13rs+26s^2$&Two-shift conjectures in \cite{Brown}\\
6&9&19&$(r^2+19rs+19s^2)/5$&Indefinite; positive-support obstruction\\
7&7&13&$(2r^2+13rs+26s^2)/3$&Positive definite; positive-support obstruction\\
8&19&31&$3r^2+31rs+93s^2$&Theorem~\ref{thm:individual}; Conjecture~\ref{conj:main}\\
9&7&2&$(r+2s)^2/2$&Singular matrix\\
10&7&97&$(4r^2+97rs+388s^2)/11$&Indefinite; positive-support obstruction\\\bottomrule
\end{tabularx}
\end{table}

Rows 4, 6 and 10 have determinant $-7/3$, $-57/5$ and $-291/11$,
respectively. Row 7 has nonintegral $B$; rows 3, 5 and 8 have positive
definite integral $B$. Thus the positive-coupling, positive-definite,
full-lattice integer-exponent requirements leave precisely $m=3,13,31$
for $m>2$, as proved in \cite[Theorem~5.1]{Brown}.
In particular, the excluded index $19$ is attached to conductor $9$;
the conductor-$19$ field supplies the surviving index $31$.

\subsection{An obstruction for any finite positive combination}
For a symmetric rational $d\times d$ matrix $B$, positive integers
$d_1,\ldots,d_d$, and $b\in\QQ^d$, put
\[
 \mathcal N_b(q)=\sum_{n\in\ZZ_{\geq0}^d}
 \frac{q^{\frac12n^{\mathsf T}Bn+b^{\mathsf T}n}}
 {\prod_{i=1}^d(q^{d_i};q^{d_i})_{n_i}}.
\]
Assume these define Puiseux series bounded below, with locally finite
coefficients. This holds, for example, when the quadratic form is
positive on the nonnegative cone away from zero.

\begin{proposition}\label{prop:support}
Suppose $\lambda_j>0$, $b_j\in\QQ^d$, and $c_j,\kappa\in\QQ$, and
\[
 \sum_{j=1}^t\lambda_jq^{c_j}\mathcal N_{b_j}(q)\in q^\kappa\RR((q)).
\]
Then every entry of $B$ is integral. More precisely, for each $j$,
\[
 c_j-\kappa\in\ZZ,\qquad b_{j,i}+B_{ii}/2\in\ZZ\quad(1\leq i\leq d).
\]
Conversely, these conditions together with integral $B$ imply the
displayed support containment whenever the series are defined as above.
\end{proposition}
\begin{proof}
Each reciprocal denominator has nonnegative coefficients and constant
term one. Therefore, for every $j,n$, the exponent
\[
 e_j(n)=\tfrac12n^{\mathsf T}Bn+b_j^{\mathsf T}n+c_j
\]
occurs with positive coefficient in the combination. No such term can
cancel, so $e_j(n)-\kappa$ is an integer. At $n=0,e_i,2e_i$, first and
second differences give $c_j-\kappa$, $b_{j,i}+B_{ii}/2$, and $B_{ii}$
integral. A mixed difference at $0,e_i,e_k,e_i+e_k$ gives $B_{ik}$
integral. Conversely,
\[
 e_j(n)-\kappa=c_j-\kappa+
 \sum_i\left(B_{ii}\binom{n_i}{2}+(b_{j,i}+B_{ii}/2)n_i\right)
 +\sum_{i<k}B_{ik}n_in_k
\]
is then integral, and the denominator expansions only add integers.
\end{proof}

\begin{corollary}\label{cor:source-support}
For each of rows 4, 6, 7 and 10 of Table~\ref{tab:source}, no nonempty
finite positive combination of rational linear shifts equals
$q^\kappa$ times an integer-power Laurent series, for any rational
$\kappa$, with the stated denominator steps and unrestricted lattice.
\end{corollary}
\begin{proof}
The entries $B_{11}$ in these four rows are $2/3$, $2/5$, $4/3$, and
$8/11$. Apply Proposition~\ref{prop:support}.
\end{proof}
This applies to any number of positive shifts. Indefiniteness alone is
only an exclusion from the positive-definite saddle framework: the
positive-cross-term sums in these rows still converge. The support
obstruction is a separate assertion about integer-spaced expansions.
Neither assertion rules out all modular functions with fractional cusp
expansions. Signed cancellation, a restricted lattice, and rescaling
$q$ require separate analysis. The positive-sum integrality obstruction
is stronger than testing each linear term separately, but is not a
general nonmodularity theorem for indefinite sums.

\subsection{Explicit low-rank precedents}
For modulus five, define
\[
 \mathcal F^{(5)}_{a,b}(q)=\sum_{r,s\geq0}
 \frac{q^{4r^2+5rs+2s^2+ar+bs}}{(q;q)_r(q;q)_s}.
\]
Storzer's identity \cite[(6.23)]{Storzer} is
$\mathcal F^{(5)}_{-1,-1}-\mathcal F^{(5)}_{-1,0}+\mathcal F^{(5)}_{0,0}=\RG(q)$.
Canceling $1-q^s$ and shifting $s$ gives
$\mathcal F^{(5)}_{-1,-1}-\mathcal F^{(5)}_{-1,0}=q\mathcal F^{(5)}_{4,3}$. Hence its positive form is
\begin{equation}\label{eq:mod5positive}
 \sum_{r,s\geq0}\frac{q^{4r^2+5rs+2s^2}(1+q^{4r+3s+1})}
 {(q;q)_r(q;q)_s}=\frac1{(q,q^4;q^5)_\infty}.
\end{equation}
Storzer discusses the failure of individual modularity for this matrix
and proves a parametric summation behind the combination
\cite[Theorem~6.3.1]{Storzer}. A different rank-two example, with
matrix $\left(\begin{smallmatrix}8&3\\3&2\end{smallmatrix}\right)$,
occurs in the Ising-character formulas of \cite{AVH}.

The modulus-seven Andrews--Gordon case \cite{Andrews,Warnaar} includes
\begin{equation}\label{eq:AGseven}
 \sum_{r,s\geq0}\frac{q^{r^2+2rs+2s^2}}{(q;q)_r(q;q)_s}
 =\frac1{(q,q^2,q^5,q^6;q^7)_\infty}.
\end{equation}
This has steps $(1,1)$ and differs from both index-$7$ row 4 and the
index-$2$ parity construction in row 1.
The nearest mixed-base comparison is
\begin{equation}\label{eq:KRnine}
 \sum_{r,s\geq0}\frac{q^{r^2+3rs+3s^2}}{(q;q)_r(q^3;q^3)_s}
 =\frac1{(q,q^3,q^6,q^8;q^9)_\infty}.
\end{equation}
Mizuno proves the three symmetric modulo-nine evaluations
\cite{MizunoKR}; Xia proves all five \cite{Xia}. Thus product modulus
nine and denominator index three coexist in this example.

\subsection{The surviving index-thirteen comparison}
To avoid confusing its shifts with \eqref{eq:family}, write
\[
 F^{(13)}_{a,b}=\sum_{r,s\geq0}
 \frac{q^{2r^2+13rs+26s^2+ar+bs}}{(q;q)_r(q^{13};q^{13})_s}.
\]
The conjectures of \cite[Conjecture~6.1]{Brown} are
\begin{align}
 F^{(13)}_{0,0}+qF^{(13)}_{2,13}
 &\stackrel?=\frac1{(q,q^5,q^8,q^{12};q^{13})_\infty},\notag\\
 F^{(13)}_{0,0}+q^6F^{(13)}_{7,26}
 &\stackrel?=\frac1{(q^2,q^3,q^{10},q^{11};q^{13})_\infty},\label{eq:thirteen}\\
 F^{(13)}_{4,13}+q^4F^{(13)}_{6,26}
 &\stackrel?=\frac1{(q^4,q^6,q^7,q^9;q^{13})_\infty}.\notag
\end{align}
The two-dimensional Fricke law of their product sides is proved in
\cite{Fricke} and recalled in \cite{Brown}. The sum evaluations are
open. At $31$ the earlier density is $8/31$, so a reciprocal product
at modulus $31$ would need eight supported residues. Since $8$ does not
divide $30$, it cannot be a multiplicative subgroup coset. The theta
target in the next section avoids this restricted product ansatz.

\section{From the theta target to the McKay--Thompson series \texorpdfstring{$T_{31A}$}{T31A}}
\label{sec:target}
Write $J_d=(q^d;q^d)_\infty$ and
\begin{equation}\label{eq:theta-target}
 \Theta_{a,b,c}(\tau)=\sum_{u,v\in\ZZ}q^{au^2+buv+cv^2},\quad
 f_{31}=\frac{\Theta_{1,1,8}-\Theta_{2,1,4}}2,\quad
 \Phi=\frac{f_{31}}{\eta(\tau)\eta(31\tau)}.
\end{equation}
The two forms have discriminant $-31$; their three reduced classes
are $(1,1,8)$ and $(2,\pm1,4)$, and the last two theta series agree.
The function $f_{31}$ is the classical weight-one cusp form on
$\Gamma_0(31)$ with character $\chi_{-31}=(\frac{-31}{\cdot})$.
It is explicitly listed in \cite[Theorem~1.1 and the class-number-three
table]{AT}. We first identify its Euler quotient with \eqref{eq:U31}.

\begin{proposition}\label{prop:thetaU}
For $|q|<1$,
\begin{equation}\label{eq:thetaU}
 f_{31}(\tau)=qJ_1J_{31}U(1,31),\qquad
 \Phi(\tau)=q^{-1/3}U(1,31).
\end{equation}
\end{proposition}
\begin{proof}
Consider the absolutely convergent signed theta sum
\[
 R(q)=\sum_{\substack{u\in\ZZ\text{ odd}\\v\in\ZZ}}
 (-1)^v q^{u^2+3uv/2+5v^2/2}.
\]
For even $v=2w$, put $x=u+w$. The exponent becomes
$x^2+xw+8w^2$, with $x-w$ odd. The omitted terms with $x-w$ even
have $x=w+2z$ and exponent
$2\{2(z+w)^2-(z+w)w+4w^2\}$.
Their series is $\Theta_{2,1,4}(2\tau)$.
For odd $v$, put $u=2z-v$. The exponent becomes
$2v^2-vz+4z^2$, with first coordinate $v$ odd. Its omitted even-first-
coordinate series is again $\Theta_{2,1,4}(2\tau)$, since
with $v=2y$,
\[
 2(2y)^2-(2y)z+4z^2=2(2z^2-zy+4y^2).
\]
The form on the right is $(2,-1,4)$, whose theta series equals
$\Theta_{2,1,4}$.
Subtracting the two parity parts gives
\begin{equation}\label{eq:Rtheta}
 R(q)=\Theta_{1,1,8}(\tau)-\Theta_{2,1,4}(\tau).
\end{equation}

For a second dissection, complete the square:
\[
 u^2+\frac32uv+\frac52v^2
 =\frac{(10v+3u)^2+31u^2}{40}.
\]
With $x=10v+3u$ and $y=u$, the new variables satisfy
$y$ odd, $x\equiv3y\pmod {10}$, with sign $(-1)^{(x-3y)/10}$.
For $y\equiv\pm1\pmod {10}$, the two classes give twice
\[
 \sum_{r,s\in\ZZ}(-1)^{r+s}
 q^{((10r-3)^2+31(10s-1)^2)/40}.
\]
For $y\equiv\pm3\pmod {10}$, they give minus twice the same expression
with $10r-3,10s-1$ replaced by $10r-1,10s-3$.
For example, in the class $y=10s-3$, take $x=-(10r-1)$; then
$(x-3y)/10=-r-3s+1$, which supplies the minus sign.
The remaining class $y\equiv5\pmod {10}$ cancels under $x\mapsto-x$:
the exponent is unchanged and the sign reverses.
Define $\vartheta_j(q)=\sum_{n\in\ZZ}(-1)^nq^{(5n^2-jn)/2}$ for $j=1,3$.
Expansion of the two squares therefore gives
\[
 R(q)=2q\{\vartheta_3(q)\vartheta_1(q^{31})-q^6\vartheta_1(q)\vartheta_3(q^{31})\}.
\]
The Jacobi triple product and the Rogers--Ramanujan identities give
$\vartheta_1(q)=J_1\RG(q)$ and $\vartheta_3(q)=J_1\RH(q)$.
Combine this with \eqref{eq:Rtheta}. Finally,
$\eta(\tau)\eta(31\tau)=q^{4/3}J_1J_{31}$ proves the second formula.
\end{proof}

This calculation is a specialization of the classical theta-dissection
setting of \cite{BY}; we include it to fix the normalization and to
make the precise target explicit. The result converts
Conjecture~\ref{conj:main} into the equivalent theta identity
\begin{equation}\label{eq:theta-conjecture}
 qJ_1J_{31}\Ssum(q)\stackrel?=f_{31}(\tau).
\end{equation}

\begin{proposition}\label{prop:Fricke}
Independently of Conjecture~\ref{conj:main},
\begin{equation}\label{eq:Fricke}
 \Phi\left(-\frac1{31\tau}\right)=\Phi(\tau).
\end{equation}
The cube $\Phi^3$ is a modular function on $\Gamma_0(31)$ and the
normalized Hauptmodul on $\Gamma_0(31)^+$ with expansion
\begin{equation}\label{eq:Haupt}
 \Phi^3=q^{-1}+3q+3q^2+6q^3+9q^4+13q^5+O(q^6).
\end{equation}
\end{proposition}
\begin{proof}
For either quadratic form $Q=(a,b,c)$, its Gram matrix
$M_Q=\left(\begin{smallmatrix}2a&b\\b&2c\end{smallmatrix}\right)$ has
determinant $31$. The matrix $31M_Q^{-1}$ is integrally equivalent to
$M_Q$ by interchanging coordinates and changing one sign.
Poisson summation gives
\[
 \Theta_Q\left(-\frac1{31\tau}\right)
 =-i\sqrt{31}\,\tau\Theta_Q(\tau).
\]
The eta product in \eqref{eq:theta-target} transforms by the same factor,
which proves \eqref{eq:Fricke}.

The usual eta-quotient criterion applies to
$\eta(\tau)^3\eta(31\tau)^3$: both congruence sums are $96$, and the
weight and character are $3$ and $\chi_{-31}$. The numerator $f_{31}^3$
has the same weight and character. Thus their quotient is modular on
$\Gamma_0(31)$ and, by \eqref{eq:Fricke}, on its Fricke extension.
Eta has no zeros on $\HH$, so the quotient has no poles there.
Direct expansion gives \eqref{eq:Haupt}. There are two cusps on
$\Gamma_0(31)$, interchanged by Fricke; the quotient has a simple pole
at each. They give one simple pole on the Fricke quotient. A nonconstant
meromorphic function with one simple pole defines a degree-one map to
the sphere, proving the Hauptmodul assertion and the stated normalization.
\end{proof}

\subsection*{Identification with \texorpdfstring{$T_{31A}$}{T31A} and the proposed new identity}
In Conway--Norton's notation \cite{CN}, the Monster series $T_{31A}$
is the normalized Hauptmodul for $\Gamma_0(31)^+$. Uniqueness and
Proposition~\ref{prop:Fricke} therefore give
\begin{equation}\label{eq:moonshine}
 T_{31A}(\tau)=\Phi(\tau)^3
 =q^{-1}\bigl(\RH(q)\RG(q^{31})-q^6\RG(q)\RH(q^{31})\bigr)^3.
\end{equation}
OEIS A058628 \cite{OEIS31} records the same coefficient formula;
its coefficient indexing begins at $q^{-1}$, accounting for the
explicit $q^{-1}$ here. Thus \eqref{eq:moonshine} is a classical
identification, not a new moonshine evaluation. The new claim is that
the positive coupled sum \eqref{eq:positive} supplies its cube root.
Conversely, the conjectural cubed identity implies
Conjecture~\ref{conj:main}: the two cube roots agree near $q=0$, and
then throughout $|q|<1$ by analytic continuation.

The related class-$93A$ formula \cite{OEIS93} uses a rescaled variable:
\begin{equation}\label{eq:moonshine93}
 T_{93A}(\tau)=\Phi(3\tau)=q^{-1}U(1,31)\big|_{q\mapsto q^3},
 \qquad T_{93A}(\tau)^3=T_{31A}(3\tau).
\end{equation}
Its first terms are $q^{-1}+q^5+q^8+q^{11}+\cdots$.
The factor and rescaling are essential when comparing the two
OEIS coefficient conventions.

Proposition~\ref{prop:Fricke} also implies that $f_{31}$ has a zero
in $\HH$: the Hauptmodul takes the value $0$, and its only cusp is
its pole. Since eta has no zeros on $\HH$, a zero of $\Phi^3$ is a
zero of $f_{31}$. A finite product or quotient
of factors $(q^r;q^N)_\infty$, $r,N>0$, has none. Thus this target is
not a single finite Euler product quotient of that form. In particular,
the shift from a residue-class product to a theta difference is a
structural feature of the candidate, not just a choice of notation.

For comparison, the definition of $U(1,31)$ immediately yields
\begin{equation}\label{eq:uncoupled}
 U(1,31)=\sum_{r,s\geq0}
 \frac{q^{r^2+31s^2}(q^r-q^{31s+6})}
 {(q;q)_r(q^{31};q^{31})_s}.
\end{equation}
This is already a rank-two expression, but its quadratic form is
uncoupled and its numerator signed. The proposed formula
\eqref{eq:positive} uses the different quadratic datum
\eqref{eq:matrices} and has a positive numerator. It is this explicit
realization, rather than the first appearance of rank two at level $31$,
that remains to be established.

\section{Nonmodularity of every individual sum}\label{sec:exclusion}
\subsection{The cubic saddle}
The determinant of $AD$ is $155$, and its leading principal entry is
positive. The symmetrizable Nahm equations therefore have a unique
solution $(x,y)\in(0,1)^2$ \cite[Section~2.1]{Mizuno}. They are
\begin{equation}\label{eq:nahm}
 1-x=x^6y,\qquad 1-y=x^{31}y^6.
\end{equation}

\begin{lemma}\label{lem:saddle}
Let $\rho$ be the root in $(0,1)$ of
\begin{equation}\label{eq:cubic}
 f(X)=X^3-5X^2+2X+1.
\end{equation}
The solution of \eqref{eq:nahm} is
\begin{equation}\label{eq:rho-sigma}
 x=\rho,\qquad
 y=\sigma=\frac{1-\rho}{\rho^6}
       =469\rho^2-2479\rho+1646.
\end{equation}
The field $K=\QQ(\rho)$ is the cyclic cubic field of conductor $19$.
\end{lemma}
\begin{proof}
Eliminate $y$ from \eqref{eq:nahm}. After cancellation of a nonzero
power of $x$, the second equation becomes
\[
 x^6+x-1-x(1-x)^6
 =-(x^2-x+1)^2(x^3-5x^2+2x+1)=0.
\]
The quadratic factor has no real zero. The derivative of $f$ has just
one zero in $(0,1)$, at which $f$ has a local maximum, while
$f(0)=1$ and $f(1)=-1$. Thus $f$ has exactly one root in $(0,1)$.
Reduction modulo $f$ gives the last expression in
\eqref{eq:rho-sigma}. For explicit rational bounds,
\[
 \frac{7781238}{10^7}<\rho<\frac{7781239}{10^7}.
\]
The function $(1-X)/X^6$ is decreasing on $(0,1)$, and its values at
these endpoints both lie in $(0,1)$.

The polynomial $f$ is irreducible by the rational-root test and has
discriminant $361=19^2$. An irreducible cubic with square discriminant
has cyclic Galois group. Its field discriminant is $19^2$ (the only
other possibility from the index formula would be $1$), and the
conductor--discriminant formula gives conductor $19$.
\end{proof}

\subsection{Radial asymptotics}
Put $d_1=1$, $d_2=31$, $(z_1,z_2)=(\rho,\sigma)$, and
$(b_1,b_2)=(a,b)$. With $q=e^{-\varepsilon}$ and
$\varepsilon\downarrow0$, the saddle expansion has the form
\begin{equation}\label{eq:expansion}
 F_{a,b}(e^{-\varepsilon})
 =\kappa e^{\Lambda/\varepsilon}
 \left(1+\beta_1(a,b)\varepsilon
       +\beta_2(a,b)\varepsilon^2+O(\varepsilon^3)\right),
 \qquad \kappa>0.
\end{equation}
We describe the normalization and the coefficients explicitly.
For $t_i=d_i\varepsilon n_i$, define
\begin{align}
 \Psi(t)&=\sum_{i=1}^2
       \frac{\pi^2/6-\Li_2(e^{-t_i})}{d_i}
       -\frac12t^{\mathsf T}D^{-1}At,\label{eq:phase}\\
 h(t)&=\exp\left(-\sum_{i=1}^2\frac{b_it_i}{d_i}\right)
              \prod_{i=1}^2(1-e^{-t_i})^{-1/2}.\label{eq:amplitude}
\end{align}
The critical point is $t_i^{(0)}=-\log z_i$. Write
\begin{equation}\label{eq:HC}
 w_i=\frac{z_i}{1-z_i},\qquad
 \mathsf H=-\Psi''(t^{(0)})
   =D^{-1}\bigl(A+\diag(w_1,w_2)\bigr),\qquad \Sigma=\mathsf H^{-1}.
\end{equation}
The matrix $\mathsf H$ is symmetric positive definite. Then
\[
 \Lambda=\Psi(t^{(0)}),\qquad
 \kappa=\frac{h(t^{(0)})}{\sqrt{\det D\,\det \mathsf H}}.
\]
In particular, there is no power of $\varepsilon$ in the leading
prefactor of \eqref{eq:expansion}.

For completeness, if
$L(z)=\Li_2(z)+\frac12\log z\log(1-z)$ is the unshifted Rogers
dilogarithm, the saddle equations give
\[
 \Lambda=\left(\frac{\pi^2}{6}-L(\rho)\right)
          +\frac1{31}\left(\frac{\pi^2}{6}-L(\sigma)\right)>0.
\]
The value recorded in \cite[Table~3, row~8]{Brown} is proved here
by the explicit five-term certificate in Appendix~\ref{app:dilog}:
\begin{equation}\label{eq:dilog-value}
 31L(\rho)+L(\sigma)=4\pi^2,\qquad \Lambda=\frac{4\pi^2}{93}.
\end{equation}
This establishes the rational leading action without relying on a
numerical identification of a Bloch-group torsion value. The
individual-sum exclusion below does not use this evaluation; its
obstruction occurs in the next two terms.

The analytic existence of \eqref{eq:expansion}, to every order, follows
from \cite[Theorem~2.1]{Mizuno} at the root of unity $1$.
The coefficient formulas below also follow directly from
the local expansion of \eqref{eq:phase}--\eqref{eq:amplitude}.
In fact, the logarithm of a summand, apart from the constant
$\log(\varepsilon\sqrt{31}/(2\pi))$, is
\begin{equation}\label{eq:local-log}
 \varepsilon^{-1}\Psi(t)+\log h(t)+\varepsilon c_0(t)
   +O(\varepsilon^3),\qquad
 c_0(t)=-\sum_{i=1}^2d_i
       \left(\frac1{24}+\frac{e^{-t_i}}{12(1-e^{-t_i})}\right).
\end{equation}
This expansion is uniform in a fixed neighborhood of $t^{(0)}$.
It follows by applying Euler--Maclaurin to each finite
$q$-Pochhammer symbol. The absence of an order-$\varepsilon^2$ term
in \eqref{eq:local-log} will be used for the second correction.

\begin{lemma}\label{lem:cusp}
If $q^cF_{a,b}$ satisfies the modularity assumptions in
Theorem~\ref{thm:individual}, then
\begin{equation}\label{eq:necessary}
 \beta_1(a,b)=c\in\QQ,\qquad
 \beta_2(a,b)-\frac12\beta_1(a,b)^2=0.
\end{equation}
\end{lemma}
\begin{proof}
Pass to the kernel of the finite-order multiplier. At the cusp $0$,
meromorphy gives a Laurent expansion in a cusp parameter of finite
width. Along $\tau=i\varepsilon/(2\pi)$, its first nonzero term has
the form
\[
 A_0e^{\lambda/\varepsilon}
 \left(1+O(e^{-\delta/\varepsilon})\right),\qquad
 A_0\ne0,\quad \delta>0.
\]
There is no algebraic correction in $\varepsilon$, since the weight
is zero. Multiplication of \eqref{eq:expansion} by
$q^c=e^{-c\varepsilon}$ gives successive coefficients
$\beta_1-c$ and $\beta_2-c\beta_1+c^2/2$.
Both must vanish, which proves \eqref{eq:necessary}.
\end{proof}

\subsection{The first correction and its rational locus}
All of the following computations take place in $K=\QQ(\rho)$, using
\begin{equation}\label{eq:reduction}
 \rho^3=5\rho^2-2\rho-1.
\end{equation}
The data in \eqref{eq:HC} reduce to
\begin{equation}\label{eq:wexplicit}
 w_1=1+4\rho-\rho^2,\qquad
 w_2=774+2546\rho-603\rho^2.
\end{equation}
\begin{equation}\label{eq:Cexplicit}
 \Sigma=\frac1{361}
 \begin{pmatrix}
 -566+949\rho-222\rho^2&3596-5611\rho+1271\rho^2\\
 3596-5611\rho+1271\rho^2&-20832+32302\rho-7099\rho^2
 \end{pmatrix}.
\end{equation}
Thus the field reduction, covariance, and all denominators needed
in the calculation are specified explicitly.

Let $\xi=(\xi_1,\xi_2)$ be a centered Gaussian vector with covariance $\Sigma$.
At the saddle define
\begin{align}
 g_i&=-\frac{b_i}{d_i}-\frac{w_i}{2},&
 h_i&=\frac{w_i(1+w_i)}2,\label{eq:gh}\\
 \psi_{3,i}&=\frac{w_i(1+w_i)}{d_i},&
 \psi_{4,i}&=-\frac{w_i(1+w_i)(1+2w_i)}{d_i},\label{eq:TV}\\
 c_0&=-\sum_{i=1}^2d_i\left(\frac1{24}+\frac{w_i}{12}\right).
 \label{eq:c0}
\end{align}
Here $g_i,h_i$ are the first and second derivatives of $\log h(t)$;
$\psi_{3,i},\psi_{4,i}$ are the pure third and fourth derivatives of $\Psi(t)$.
Set
\[
 L_1=\sum_i(g_i\xi_i+\psi_{3,i}\xi_i^3/6),\qquad
 L_2=c_0+\sum_i(h_i\xi_i^2/2+\psi_{4,i}\xi_i^4/24).
\]
Expanding \eqref{eq:local-log} at
$t=t^{(0)}+\sqrt{\varepsilon}\xi$ gives
\begin{equation}\label{eq:beta1gauss}
 \beta_1=\EE(L_2+L_1^2/2).
\end{equation}
Wick's formula expresses this as
\begin{align}
 \beta_1={}&c_0+\frac12\sum_i h_i\Sigma_{ii}
       +\frac18\sum_i \psi_{4,i}\Sigma_{ii}^2
       +\frac12\sum_{i,j}g_ig_j\Sigma_{ij}\notag\\
 &+\frac12\sum_{i,j}g_i\psi_{3,j}\Sigma_{ij}\Sigma_{jj}
       +\frac1{72}\sum_{i,j}\psi_{3,i}\psi_{3,j}
          \left(9\Sigma_{ii}\Sigma_{jj}\Sigma_{ij}+6\Sigma_{ij}^3\right).
 \label{eq:beta1contracted}
\end{align}

\begin{proposition}\label{prop:first}
The coefficient $\beta_1$ is
\begin{equation}\label{eq:beta1explicit}
 \beta_1(a,b)=B_0(a,b)
       +\frac{P(a,b)\rho-Q(a,b)\rho^2}{22382},
\end{equation}
where
\begin{align}
 B_0&=-\frac{
 26319a^2-10788ab+151404a+1008b^2-28179b+161107}{33573},
 \label{eq:B0}\\
 P&=29419a^2-11222ab+154907a+1042b^2-28210b+162378,
 \label{eq:P}\\
 Q&=6882a^2-2542ab+28830a+229b^2-5239b+32116.
 \label{eq:Q}
\end{align}
For $a,b\in\QQ$, the number $\beta_1(a,b)$ is rational if and only if
\begin{equation}\label{eq:unique}
 (a,b)=\left(\frac{19}{3},\frac{124}{3}\right),
 \qquad \beta_1=\frac{8806}{3249}.
\end{equation}
\end{proposition}
\begin{proof}
Substitute \eqref{eq:wexplicit}--\eqref{eq:c0} into
\eqref{eq:beta1contracted} and reduce with \eqref{eq:reduction}.
The respective coefficients of $1,\rho,\rho^2$ are
$B_0,P/22382,-Q/22382$, giving \eqref{eq:beta1explicit}.
This is a quadratic polynomial calculation over $\QQ$.
Since $1,\rho,\rho^2$ are linearly independent over $\QQ$,
rationality is equivalent to $P=Q=0$.

Write $P=p_2b^2+p_1b+p_0$ and $Q=q_2b^2+q_1b+q_0$.
The quadratic resultant formula gives the explicit elimination
\begin{align}
 \Res_b(P,Q)
 &=(p_2q_0-p_0q_2)^2
   -(p_2q_1-p_1q_2)(p_1q_0-p_0q_1)\notag\\
 &=-961(3a-19)\mathcal R(a),\label{eq:resultant}\\
 \mathcal R(a)
 &=1620529a^3-30790051a^2-202559760a+3789270004.
 \label{eq:eliminant}
\end{align}
The reduction of $\mathcal R$ modulo $3$ is $a^3-a^2+1$.
Its values at $0,1,2$ are nonzero in $\mathbb F_3$, so it is
irreducible over $\mathbb F_3$, and hence $\mathcal R$ is irreducible
over $\QQ$. Therefore a common rational zero of $P,Q$ must have
$a=19/3$. At this value,
\begin{align}
 P(19/3,b)&=\frac29(3b-124)(1563b-84320),\notag\\
 Q(19/3,b)&=\frac13(3b-124)(229b-11873).
 \label{eq:specialized}
\end{align}
The other two roots differ, so $b=124/3$. Substitution in
\eqref{eq:B0} gives the stated value of $\beta_1$.
\end{proof}

In particular, integer linear terms are already excluded by the first
condition in Lemma~\ref{lem:cusp}. The rational pair
\eqref{eq:unique} requires the next coefficient.

\subsection{The second correction and completion of the proof}
We now fix \eqref{eq:unique}. In addition to
\eqref{eq:gh}--\eqref{eq:c0}, put
\begin{align}
 j_i&=-\frac{w_i(1+w_i)(1+2w_i)}2,&
 k_i&=\frac{w_i(1+w_i)(1+6w_i+6w_i^2)}2,\label{eq:jk}\\
 \psi_{5,i}&=\frac{w_i(1+w_i)(1+6w_i+6w_i^2)}{d_i},&
 \psi_{6,i}&=-\frac{w_i(1+w_i)(1+14w_i+36w_i^2+24w_i^3)}{d_i},
 \label{eq:WZ}\\
 c_i'&=\frac{d_iw_i(1+w_i)}{12},&
 c_i''&=-\frac{d_iw_i(1+w_i)(1+2w_i)}{12}.
 \label{eq:cderivatives}
\end{align}
The quantities $j_i,k_i$ are the third and fourth derivatives of
$\log h(t)$; $\psi_{5,i},\psi_{6,i}$ are the fifth and sixth pure derivatives of
$\Psi(t)$. The last line contains the derivatives of $c_0(t)$.
Define
\begin{align}
 L_3&=\sum_i(c_i'\xi_i+j_i\xi_i^3/6+\psi_{5,i}\xi_i^5/120),\notag\\
 L_4&=\sum_i(c_i''\xi_i^2/2+k_i\xi_i^4/24+\psi_{6,i}\xi_i^6/720).
 \label{eq:L34}
\end{align}

\begin{lemma}\label{lem:second}
At the pair \eqref{eq:unique},
\begin{equation}\label{eq:beta2value}
 \beta_2=\frac{33813353}{10556001},\qquad
 \beta_2-\frac12\beta_1^2=-\frac{1653155}{3518667}.
\end{equation}
\end{lemma}
\begin{proof}
The expansion of the exponential of \eqref{eq:local-log} gives
\begin{equation}\label{eq:beta2gauss}
 \beta_2=
 \EE\left(L_4+L_1L_3+\frac12L_2^2
                 +\frac12L_1^2L_2+\frac1{24}L_1^4\right).
\end{equation}
Indeed, after removal of the leading Gaussian, the exponent is
$\sqrt\varepsilon L_1+\varepsilon L_2+
\varepsilon^{3/2}L_3+\varepsilon^2L_4+\cdots$.
Odd Gaussian moments vanish. The five terms in
\eqref{eq:beta2gauss} are exactly the remaining contributions of
order $\varepsilon^2$.

Here is a finite algebraic specification of the evaluation.
At \eqref{eq:unique} the two linear coefficients are
\[
 g_1=\frac{\rho^2}{2}-2\rho-\frac{41}{6},\qquad
 g_2=\frac{603\rho^2}{2}-1273\rho-\frac{1165}{3}.
\]
Use these values, \eqref{eq:wexplicit}--\eqref{eq:Cexplicit}, and
\eqref{eq:gh}--\eqref{eq:L34}.
For a polynomial $p(u,v)$ let
\[
 \mathfrak E_\Sigma(p)=
 \left.
 \sum_{\ell=0}^{6}\frac1{2^\ell\ell!}
 \left(\Sigma_{11}\partial_u^2+
       2\Sigma_{12}\partial_u\partial_v+
       \Sigma_{22}\partial_v^2\right)^\ell p(u,v)
 \right|_{u=v=0}.
\]
This is Gaussian expectation for polynomials of degree at most twelve,
which includes every term of \eqref{eq:beta2gauss}.
Consequently the exact arithmetic identity used here is
\begin{equation}\label{eq:certificate}
 \mathfrak E_\Sigma\left(
 L_4+L_1L_3+\tfrac12L_2^2+
 \tfrac12L_1^2L_2+\tfrac1{24}L_1^4
 \right)
 =\frac{33813353}{10556001}\qquad\text{in }K.
\end{equation}
Here $\xi_1=u$ and $\xi_2=v$. All entries are given explicitly above; only polynomial
multiplication, differentiation, and the reduction
\eqref{eq:reduction} are involved. An equivalent way to evaluate the
left side is to use $M_{00}=1$, zero moments of odd total degree,
and the recurrence
\begin{align}
 M_{r,s}&=(r-1)\Sigma_{11}M_{r-2,s}
              +s\Sigma_{12}M_{r-1,s-1}\quad(r>0),\notag\\
 M_{0,s}&=(s-1)\Sigma_{22}M_{0,s-2}\quad(s>0),
 \label{eq:momentrecurrence}
\end{align}
where negative indices give zero.
The reduced coefficients of $1,\rho,\rho^2$ in the left side of
\eqref{eq:certificate} are respectively
$33813353/10556001$, $0$, and $0$.
Finally, $\beta_1=8806/3249$ and $3249^2=10556001$ give
the second identity in \eqref{eq:beta2value}.
\end{proof}

\begin{proof}[Proof of Theorem~\ref{thm:individual}]
If $q^cF_{a,b}$ were modular as stated, Lemma~\ref{lem:cusp}
would make $\beta_1(a,b)$ rational. Proposition~\ref{prop:first}
then forces \eqref{eq:unique} and $c=8806/3249$.
Lemma~\ref{lem:second} shows that the second necessary condition
in \eqref{eq:necessary} fails at this pair.
\end{proof}

The role of the constant term $c$ is now explicit. It removes the
first logarithmic correction, but cannot change the second one.
Thus the last nonzero rational number in \eqref{eq:beta2value}
is an obstruction to every rational normalization.

\begin{corollary}\label{cor:vector}
Let $\mathbf f$ be a weight-zero vector-valued modular function on a
finite-index subgroup of $\mathrm{SL}_2(\ZZ)$, meromorphic at its cusps,
whose multiplier representation has finite image. No component of
$\mathbf f$ can be $q^cF_{a,b}$ with $a,b,c\in\QQ$.
\end{corollary}
\begin{proof}
The kernel of a finite-image representation has finite index. Each
component is a scalar modular function on this kernel, so
Theorem~\ref{thm:individual} applies.
\end{proof}

The corollary holds regardless of vector dimension. The following
argument removes the finite-image hypothesis when only a constant
Fricke transformation is requested.

\subsection{A direct obstruction to finite Fricke closure}
Let $\mathcal M_\infty$ denote the holomorphic functions on $\HH$ which,
for sufficiently large $\operatorname{Im}\tau$, have a convergent
expansion $\sum_{n\geq n_0}a_n e^{2\pi i n\tau/h_\infty}$ for some positive
integer $h_\infty$ and integer $n_0$. Thus meromorphy is required in a
finite-width Fourier parameter, not merely along the imaginary axis.
Set $W_N\tau=-1/(N\tau)$ for $N>0$.

\begin{lemma}\label{lem:fricke-radial}
Suppose that $g$ is a nonzero holomorphic function on $\HH$,
$g\circ W_N\in\mathcal M_\infty$, and
\[
 g\left(\frac{i\varepsilon}{2\pi}\right)
 =\kappa e^{\lambda/\varepsilon}
 (1+d_1\varepsilon+d_2\varepsilon^2+O(\varepsilon^3)),
 \qquad \kappa\ne0,
\]
with real $\lambda$. Then $d_1=d_2=0$.
\end{lemma}
\begin{proof}
Write $h=g\circ W_N$ and let $a_{n_*}e^{2\pi i n_*\tau/h_\infty}$
be its first nonzero Fourier term. Since $W_N^2=1$,
\[
 g\left(\frac{i\varepsilon}{2\pi}\right)
 =a_{n_*}e^{-4\pi^2 n_*/(Nh_\infty\varepsilon)}
 \left(1+O(e^{-4\pi^2/(Nh_\infty\varepsilon)})\right).
\]
Comparison of the leading exponent and amplitude, followed by the
coefficients of $\varepsilon$ and $\varepsilon^2$, proves the claim.
\end{proof}

\begin{corollary}[Fricke exclusion]\label{cor:fricke}
For every $a,b,c\in\QQ$ and $N>0$, the Fricke image of
$q^cF_{a,b}$ does not belong to $\mathcal M_\infty$.
Consequently, no finite-dimensional subspace of $\mathcal M_\infty$
preserved by $W_N$ contains $q^cF_{a,b}$. Equivalently, no finite vector
with components in $\mathcal M_\infty$ and with one such component
satisfies a constant-matrix law
\[
 \mathbf G(W_N\tau)=S\mathbf G(\tau).
\]
\end{corollary}
\begin{proof}
For $g=q^cF_{a,b}$, Lemma~\ref{lem:fricke-radial} and
\eqref{eq:expansion} require $\beta_1=c\in\QQ$ and
$\beta_2-\beta_1^2/2=0$. Proposition~\ref{prop:first} and
Lemma~\ref{lem:second} show that these requirements are incompatible.
A preserved space or the displayed matrix law would make
$g\circ W_N$ a finite linear combination of members of
$\mathcal M_\infty$, which again belongs to $\mathcal M_\infty$.
\end{proof}

The dimension and the entries of $S$ are unrestricted. The cusp
assumption is essential: pairing any holomorphic $g$ with
$g\circ W_N$ gives a formal two-component exchange law, but its
components need not be meromorphic at infinity. The corollary concerns
individual sums. It does not obstruct the Fricke-invariant theta
target of Proposition~\ref{prop:Fricke} or establish the failure of
Fricke symmetry for the positive combination.

\section{A positive combination surviving the asymptotic tests}
\label{sec:combination}
\subsection*{Why this matrix is worth testing}
Positive definiteness does not single out $31$. For example,
\[
 B_m=\begin{pmatrix}6&m\\m&6m\end{pmatrix},\qquad
 \det B_m=m(36-m),
\]
is positive definite for every integer $1\leq m\leq35$.
The extra input for $m=31$ is the two-term identity
$31L(\rho)+L(\sigma)=4\pi^2$, proved in Appendix~\ref{app:dilog}.
It evaluates the common saddle action as $\Lambda=4\pi^2/93$,
so the leading growth has the rational dilogarithmic value required
by a modular cusp expansion. The following observation explains why
this test precedes a search over positive combinations.

\begin{lemma}[Persistence of the leading action]\label{lem:persistence}
Let finitely many functions satisfy
$G_j(e^{-\varepsilon})=\kappa_j e^{\Lambda/\varepsilon}
(1+O(\varepsilon))$ with $\kappa_j>0$ and real $\Lambda$.
If $\lambda_j>0$ and $t_j\in\QQ$, then
\[
 \sum_j\lambda_jq^{t_j}G_j(q)
 =\left(\sum_j\lambda_j\kappa_j\right)
 e^{\Lambda/\varepsilon}(1+O(\varepsilon)),\qquad q=e^{-\varepsilon}.
\]
If this combination is a weight-zero modular function on a finite-index
subgroup with finite-order multiplier, meromorphic at the cusps, then
$\Lambda/\pi^2\in\QQ$.
\end{lemma}
\begin{proof}
Each factor $q^{t_j}=1+O(\varepsilon)$ and the displayed amplitude
is positive. At the cusp $0$, after passing to the multiplier kernel,
a first Fourier term of order $n_*$ and width $h_\infty$ contributes
$\exp(-4\pi^2n_*/(h_\infty\varepsilon))$.
Comparison gives $\Lambda=-4\pi^2n_*/h_\infty$.
\end{proof}

The same conclusion holds for signed coefficients when the sum of the
weighted leading amplitudes is nonzero. Thus shifts can cancel higher
radial corrections without changing this leading action. The
dilogarithm identity motivates the search at $31$; it is not needed
for the contiguous rearrangement or for calculating the correction
equations below. Its rational value is only a necessary test, not a
sufficient modularity criterion. Within the specified cubic source,
the matrix restrictions already reduce the choices to $3,13,31$.

\subsection{Contiguous identities and positivity}
The following identity holds before any modular evaluation is assumed.
\begin{proposition}\label{prop:positive}
The signed expression \eqref{eq:signed} equals the positive expression
\eqref{eq:positive}. In particular, $\Ssum\in\ZZ_{\geq0}[[q]]$.
\end{proposition}
\begin{proof}
Canceling $1-q^r$ and replacing $r$ by $r+1$ gives
\begin{equation}\label{eq:contig}
 F_{a,b}-F_{a+1,b}=q^{a+3}F_{a+6,b+31}.
\end{equation}
Successive applications give
\[
\begin{aligned}
 q^2F_{3,31}-q^{22}F_{16,93}
 &=q^2F_{3,31}-q^9F_{10,62}+q^9F_{11,62}\\
 &=q^2F_{3,31}-q^2F_{4,31}+q^2F_{5,31}+q^9F_{11,62}\\
 &=q^8F_{9,62}+q^2F_{5,31}+q^9F_{11,62}.
\end{aligned}
\]
Adding $F_{0,0}$ proves the assertion.
\end{proof}
The rearrangement is not specific to $31$: if the quadratic exponent
is $3r^2+mrs+3ms^2$, the same calculation gives
$F^{(m)}_{a,b}-F^{(m)}_{a+1,b}=q^{a+3}F^{(m)}_{a+6,b+m}$.
Hence the signed-to-positive identity holds for every positive integer
$m$, with all multiples of $31$ replaced by the corresponding
multiples of $m$. It is the proposed modular evaluation, rather than
this contiguous relation, that distinguishes the candidate.
Equivalently, Conjecture~\ref{conj:main} proposes the single displayed sum
\begin{equation}\label{eq:positive-numerator}
 U(1,31)\stackrel?=
 \sum_{r,s\geq0}\frac{q^{3r^2+31rs+93s^2}
 \bigl(1+q^{5r+31s+2}+q^{9r+62s+8}+q^{11r+62s+9}\bigr)}
 {(q;q)_r(q^{31};q^{31})_s}.
\end{equation}
Each term of this numerator has coefficient one. Thus signed
cancellation in \eqref{eq:signed} is compatible with a manifestly
positive representation, as in the modulus-five comparison
\eqref{eq:mod5positive}.

\subsection{Necessary conditions for a combination}
For $b=31k$, the leading amplitude in \eqref{eq:expansion} is
$\kappa_0\rho^a\sigma^k$, where $\kappa_0$ is independent of $a,k$.
Consider $\sum_jc_jq^{t_j}F_{a_j,31k_j}$ and put
\[
 \omega_j=\rho^{a_j}\sigma^{k_j},\qquad
 \beta_{r,j}=\beta_r(a_j,31k_j),\qquad \beta_{0,j}=1.
\]
Assume its leading amplitude is nonzero. A common rational
$q$-normalization may be incorporated into the $t_j$. The first two
necessary weight-zero modularity conditions are then
\begin{align}
 \sum_jc_j\omega_j(\beta_{1,j}-t_j)&=0,\label{eq:test1}\\
 \sum_jc_j\omega_j(\beta_{2,j}-t_j\beta_{1,j}+t_j^2/2)&=0.
 \label{eq:test2}
\end{align}
Unlike the individual test, these equations allow the irrational parts
of different constituents to cancel. They explain the logical gap
between Theorem~\ref{thm:individual} and the combination problem.

For three shifts with linearly independent amplitudes in
$K=\QQ(\rho)$, let $M_0,M_1,M_2$ be the rational matrices with columns
the coordinates of $\omega_j,\omega_j\beta_{1,j},\omega_j\beta_{2,j}$ in the basis
$1,\rho,\rho^2$. Write $\mathsf R_1=M_0^{-1}M_1$ and $\mathsf R_2=M_0^{-1}M_2$.
If all $c_j$ are nonzero, \eqref{eq:test1} gives
$t_j=(\mathsf R_1c)_j/c_j$. Substitution into \eqref{eq:test2} yields
\begin{equation}\label{eq:search-quadrics}
 2c_j\bigl((\mathsf R_2-\mathsf R_1^{\,2})c\bigr)_j
       +(\mathsf R_1c)_j^2=0,\qquad j=1,2,3.
\end{equation}
Indeed, in coordinates the second equation is
$M_2c-M_1\diag(t)c+\tfrac12M_0\diag(t)^2c=0$; multiplication by
$M_0^{-1}$ and use of $\diag(t)c=\mathsf R_1c$ gives
\eqref{eq:search-quadrics}.

For the shifts $(a,k)=(0,0),(3,1),(16,3)$, exact elimination in the
chart $c=(1,u,v)$ gives the reduced Gr\"obner basis $u-1,v+1$.
The corresponding normalizations are
\begin{equation}\label{eq:search-solution}
 c=(1,1,-1),\qquad t=(-1/3,5/3,65/3).
\end{equation}
This recovers $q^{-1/3}\Ssum(q)$ from the asymptotic equations, before
the theta target is imposed.

\subsection{Scope of the finite search}
The positive searches used the $105$ shifts
$0\leq a\leq20$, $b=31k$, $0\leq k\leq4$.
The first two exact tests excluded all $5,460$ distinct pairs and
$187,460$ distinct triples with equal positive coefficients.
Allowing primitive positive coefficient tuples with entries in
$\{1,2,3\}$ gave $38,220$ weighted pairs and $4,686,500$ weighted
triples, with no survivor. These tests allow arbitrary rational $t_j$;
the bounds concern the linear shifts and the specified coefficient sets.
A separate elimination excluded all nonzero rational coefficient ratios
for pairs in this box. These are finite computational findings.
The first screen uses the prime $1000000007$:
a nonzero reduced obstruction certifies a nonzero rational obstruction
when its denominators are invertible. In these positive searches,
singular amplitude matrices and every unrejected case are passed to
exact rational elimination; they
are not discarded by the finite-field screen.

For the signed search, order the
shifts lexicographically and test the sign patterns $(1,1,-1)$,
$(1,-1,1)$, and $(1,-1,-1)$. Of the $187,460$ triples in this box,
$187,157$ have independent amplitudes and $303$ have dependent
amplitudes. For each full-rank triple and sign pattern, the first
condition fixes the rational $t_j$; the second condition then gives
$561,471$ tests. Exactly one survives the finite-field screen and
exact follow-up: $(a,k)=(0,0),(3,1),(16,3)$ with $c=(1,1,-1)$ and
\eqref{eq:search-solution}. These counts refer to a fresh reconstruction of the signed search;
they are not historical counts for the initial discovery.

The arbitrary-coefficient Gr\"obner elimination
\eqref{eq:search-quadrics} was completed for exactly one full-rank
triple, namely this survivor; it gives $u-1,v+1$ as stated above.
The sign-pattern search is not an arbitrary-rational-coefficient
classification, and its $303$ dependent triples are not classified.
Accordingly, uniqueness is asserted only within the stated full-rank
sign-pattern search. Neither global uniqueness nor minimality of four
positive constituents is asserted.

\section{Exact evidence for Conjecture~\ref{conj:main}}\label{sec:evidence}
The conjecture tested in this section is the level-$31$ sum--theta
identity in Conjecture~\ref{conj:main}:
\[
 \Ssum(q)=F_{0,0}(q)+q^2F_{3,31}(q)-q^{22}F_{16,93}(q)=U(1,31).
\]
Equivalently, the positive combination in \eqref{eq:positive} is
conjectured to equal $U(1,31)$. We examine its coefficients and radial
asymptotics below.

\subsection{Coefficient agreement}
Exact integer arithmetic gives
\begin{equation}\label{eq:finitecheck}
 \Ssum(q)-\frac{f_{31}(\tau)}{qJ_1J_{31}}=O(q^{2001}).
\end{equation}
Both the signed and positive expressions were computed independently
and agreed through this bound. The first terms are
\[
\begin{split}
 \Ssum(q)={}&1+q^2+q^3+q^4+q^5+q^6+q^7\\
 &+2q^8+2q^9+2q^{10}+2q^{11}+3q^{12}+3q^{13}+4q^{14}+\cdots.
\end{split}
\]
The comparison can be reproduced using only finite coefficient
recurrences. For $d=1,31$, the coefficients of $(q^d;q^d)_n^{-1}$
are generated by successively multiplying by $(1-q^{dj})^{-1}$.
For the sum side only pairs with
$3r^2+31rs+93s^2+ar+bs$ below the truncation bound contribute.
After multiplication by $qJ_1J_{31}$, the coefficients are compared
with the finite enumeration of
$u^2+uv+8v^2$ and $2u^2+uv+4v^2$.
Both forms are at least $(u^2+v^2)/2$, so a finite square of
explicit radius suffices. There is no numerical tolerance in this check.

\subsection{The exact leading amplitude}
\begin{proposition}\label{prop:amplitude}
The normalized combination has leading radial amplitude one:
\begin{equation}\label{eq:leading}
 e^{\varepsilon/3}\Ssum(e^{-\varepsilon})
 \sim e^{4\pi^2/(93\varepsilon)}
 \qquad(\varepsilon\downarrow0).
\end{equation}
\end{proposition}
\begin{proof}
From \eqref{eq:amplitude}--\eqref{eq:HC},
\[
 \kappa_0^2=\frac1{31\det\mathsf H(1-\rho)(1-\sigma)}.
\]
The signed expression has leading amplitude factor
\[
 W=1+\rho^3\sigma-\rho^{16}\sigma^3=-55+85\rho-16\rho^2.
\]
Reduction by \eqref{eq:reduction} gives
\begin{equation}\label{eq:amplitude-certificate}
 W^2=31\det\mathsf H(1-\rho)(1-\sigma).
\end{equation}
The positive expression \eqref{eq:positive} also gives
$W=1+\rho^5\sigma+\rho^9\sigma^2+\rho^{11}\sigma^2>0$.
Thus $\kappa_0W=1$. Appendix~\ref{app:dilog} establishes the common
action $\Lambda=4\pi^2/93$, proving \eqref{eq:leading}.
\end{proof}
Independently, Proposition~\ref{prop:Fricke} and the leading term
$\Phi=q^{-1/3}(1+O(q))$ at infinity give
$\Phi(i\varepsilon/(2\pi))\sim e^{4\pi^2/(93\varepsilon)}$.
Hence both the action and the leading amplitude agree exactly.

\subsection{Six vanishing corrections}
The exact local Gaussian expansion gives, for
\[
 (a_j,b_j,t_j,c_j)=(0,0,-1/3,1),\quad
 (3,31,5/3,1),\quad(16,93,65/3,-1),
\]
the identities
\begin{equation}\label{eq:six}
 \sum_j c_j\rho^{a_j}\sigma^{b_j/31}
 \sum_{r=0}^n\frac{(-t_j)^{n-r}}{(n-r)!}\beta_r(a_j,b_j)=0,
 \qquad 1\leq n\leq6.
\end{equation}
All three rational coordinates in $K$ vanish at each order.
Together with Proposition~\ref{prop:amplitude}, this yields the checked
expansion
\[
 e^{-4\pi^2/(93\varepsilon)+\varepsilon/3}\Ssum(e^{-\varepsilon})
 =1+O(\varepsilon^7).
\]

For a finite specification of the calculation, extend
\eqref{eq:local-log} to
\[
 \varepsilon^{-1}\Psi(t)+\log h(t)+\varepsilon c_0(t)
 +\varepsilon^3c_3(t)+\varepsilon^5c_5(t)+O(\varepsilon^7),
\]
where
\[
 c_3(t)=\sum_i\frac{d_i^3}{720}\Li_{-2}(e^{-t_i}),\qquad
 c_5(t)=-\sum_i\frac{d_i^5}{30240}\Li_{-4}(e^{-t_i}).
\]
At $t=t^{(0)}+\sqrt\varepsilon\xi$, remove the leading Gaussian and
write the remaining exponent as $\sum_{j\geq1}L_j\varepsilon^{j/2}$.
Its exponential coefficients obey
\begin{equation}\label{eq:En}
 E_0=1,\qquad E_n=\frac1n\sum_{j=1}^n jL_jE_{n-j},\qquad
 \beta_n=\EE(E_{2n}).
\end{equation}
Expansion through order $\varepsilon^6$ uses derivatives of $\Psi$
through order fourteen, of $\log h$ through order twelve, and of
$c_0,c_3,c_5$ through orders ten, six, and two, respectively.
Together with the moment recurrence \eqref{eq:momentrecurrence},
these determine all entries of \eqref{eq:six}.
All derivatives are rational functions of $w_i$, since
$\partial_{t_i}=-w_i(1+w_i)\partial_{w_i}$.
For clarity, the resulting field coordinates are
\[
\begin{array}{c|ccc}
 n&[1]&[\rho]&[\rho^2]\\\hline
 0&-55&85&-16\\
 1,2,3,4,5,6&0&0&0
\end{array}
\]
before dividing by $W$. The same expansion procedure reproduced the
three proved modulus-nine controls through order four, and one through
order six. The finite vanishing statements in \eqref{eq:six} do not
establish the required vanishing at every order.

\section{A dual companion and conjectural Fricke symmetry}
\label{sec:dual-fricke}
The dual datum reveals a further proposed connection between the
level-$31$ sums and $T_{31A}$. We first define its components, then
state the Fricke law in full and explain its action without matrix
notation. The transformation of the sums remains conjectural; the
translation law and the matrix calculations below are unconditional.

\subsection{The dual sum and its three components}
The inverse of the matrix in \eqref{eq:matrices} is
\[
 A^{-1}=\frac15\begin{pmatrix}6&-1\\-31&6\end{pmatrix},
 \qquad \det A=5.
\]
With the same denominator steps, define
\begin{equation}\label{eq:dual-family}
 F^{\vee}_{\alpha,\beta}(q)=
 \sum_{r,s\geq0}
 \frac{q^{(3r^2-31rs+93s^2)/5+\alpha r+\beta s}}
 {(q;q)_r(q^{31};q^{31})_s},\qquad \alpha,\beta\in\QQ.
\end{equation}
Its quadratic form is $\tfrac12(r,s)A^{-1}D(r,s)^{\mathsf T}$.
It is positive definite, so these sums converge locally uniformly on
$\HH$. Fractional powers retain the convention
$q^c=e^{2\pi i c\tau}$. Consider the particular combination
\begin{equation}\label{eq:dual-combination}
 S^*(\tau)=q^{-1}F^{\vee}_{0,0}(q)
 +q^{12/5}F^{\vee}_{-13/5,\,93/5}(q)
 -q^{2/5}F^{\vee}_{3/5,\,62/5}(q).
\end{equation}
Its three linear-shift pairs are the images of
$(0,0),(3,31),(16,93)$ from \eqref{eq:signed} under
$(a,b)\mapsto((6a-b)/5,(-31a+6b)/5)$.

Every exponent in $S^*$ belongs to one of the classes
$0,2/5,3/5$ modulo $\ZZ$. Indeed, put $u=r+4s$. Five times
the complete numerator exponent in the three respective summands of
\eqref{eq:dual-combination} is congruent modulo $5$ to
$3u^2$, $3(u+2)^2$, and $3(u+3)^2$. Each takes values in
$\{0,2,3\}$, and expansion of the denominators adds only integer
powers of $q$. Thus there is a unique decomposition
\begin{equation}\label{eq:dual-splitting}
 S^*=S_0+S_2+S_3,\qquad S_j\in q^{j/5}\ZZ((q)).
\end{equation}
The subscript labels the fractional exponent class; it is neither a
power nor a derivative. The initial terms are
\begin{equation}\label{eq:dual-initial}
\begin{aligned}
 S_0&=q^{-1}+1+q+3q^2+4q^3+5q^4+\cdots,\\
 S_2&=2q^{7/5}+4q^{12/5}+4q^{17/5}+8q^{22/5}+\cdots,\\
 S_3&=2q^{-2/5}+2q^{3/5}+4q^{8/5}+4q^{13/5}+\cdots.
\end{aligned}
\end{equation}
In particular all three components are nonzero.
The number $31$ comes from the second denominator step, whereas $5$
enters through $\det A$ and the dual exponent. These distinct roles
make fifth-root constants natural in the dual problem, but do not
derive the proposed transformation.

\subsection{The Fricke conjecture, written out}
Write $W_{31}\tau=-1/(31\tau)$ and collect the components as
\[
 V(\tau)=(S_0(\tau),S_2(\tau),S_3(\tau))^{\mathsf T}.
\]

\begin{conjecture}[Fricke symmetry of the dual]\label{conj:dual-fricke}
For every $\tau\in\HH$,
\begin{equation}\label{eq:dual-fricke}
 \begin{pmatrix}
 S_0\!\left(-\dfrac1{31\tau}\right)\\[3pt]
 S_2\!\left(-\dfrac1{31\tau}\right)\\[3pt]
 S_3\!\left(-\dfrac1{31\tau}\right)
 \end{pmatrix}
 =\underbrace{\frac1{\sqrt5}
 \begin{pmatrix}
 1&1&1\\
 2&2\cos(2\pi/5)&2\cos(4\pi/5)\\
 2&2\cos(4\pi/5)&2\cos(2\pi/5)
 \end{pmatrix}}_{\displaystyle\mathbf M_{31}}
 \begin{pmatrix}S_0(\tau)\\S_2(\tau)\\S_3(\tau)\end{pmatrix}.
\end{equation}
\end{conjecture}

The three equivalent scalar statements make the action simpler:
\begin{equation}\label{eq:dual-exchange}
\begin{aligned}
 S_0(W_{31}\tau)&=S^*(\tau)/\sqrt5,\\
 S^*(W_{31}\tau)&=\sqrt5\,S_0(\tau),\\
 (S_2-S_3)(W_{31}\tau)&=(S_2-S_3)(\tau).
\end{aligned}
\end{equation}
The first is the first row of \eqref{eq:dual-fricke}; the other two
follow by summing all rows and subtracting the last two rows.
Conversely, these three statements determine each transformed
component. Thus Fricke would exchange $S_0$ and $S^*/\sqrt5$ and
fix $S_2-S_3$. The full sum $S^*$ is not asserted to be
Fricke-invariant. These are weight-zero laws, with no extra power of
$\tau$. Equivalently, the combinations
$S^*+\sqrt5S_0$, $S_2-S_3$, and $S^*-\sqrt5S_0$ would have
Fricke eigenvalues $+1,+1,-1$.

The law was tested at twelve points: $\tau=iy$ for
$y=0.14,0.16,\ldots,0.30$, and
$0.04+0.18i$, $-0.07+0.23i$, $0.12+0.19i$.
Using the dual expansion through $q^{520}$ at $90$-digit precision,
the largest observed componentwise residual
$|a-b|/\max(1,|a|,|b|)$ was below $1.6\times10^{-90}$.
These are floating-point residuals for truncated evaluations, not
certified bounds for the infinite-series error.

\subsection{Exact translation and icosahedral matrix algebra}
Unlike the Fricke law, translation follows directly from
\eqref{eq:dual-splitting}. With $\zeta_5=e^{2\pi i/5}$, one has
\begin{equation}\label{eq:dual-translation}
 V(\tau+1)=\mathbf R_5V(\tau),\qquad
 \mathbf R_5=\diag(1,\zeta_5^2,\zeta_5^3).
\end{equation}

\begin{proposition}\label{prop:dual-matrices}
The matrices in \eqref{eq:dual-fricke} and
\eqref{eq:dual-translation} satisfy
\begin{equation}\label{eq:dual-matrix-relations}
 \mathbf M_{31}^{\,2}=I,\qquad \mathbf R_5^{\,5}=I,\qquad
 (\mathbf M_{31}\mathbf R_5)^3=-I.
\end{equation}
They generate a group of order $120$, isomorphic to
$A_5\times\{\pm I\}$, and both preserve $x^2+yz$.
\end{proposition}
\begin{proof}
Put $\phi=(1+\sqrt5)/2$. The two trigonometric entries of
$\sqrt5\mathbf M_{31}$ are $\phi-1$ and $-\phi$.
Direct multiplication, using $\phi^2=\phi+1$ and
$1+\zeta_5+\cdots+\zeta_5^4=0$, with
$\zeta_5+\zeta_5^{-1}=\phi-1$, proves
\eqref{eq:dual-matrix-relations}; it also gives
$\det\mathbf M_{31}=-1$. The matrices
$J=-\mathbf M_{31}$ and $\mathbf R_5$ therefore satisfy
$J^2=\mathbf R_5^5=(J\mathbf R_5)^3=I$.
The spherical $(2,3,5)$ presentation gives $A_5$; its nontrivial
image here is faithful because $A_5$ is simple.
Both generators have determinant one, whereas $-I$ does not.
Adjoining $-I=(\mathbf M_{31}\mathbf R_5)^3$ gives the stated
group of order $120$. Finally, direct substitution shows that
$x^2+yz$ is preserved by each matrix.
\end{proof}
This is the full icosahedral group in its three-dimensional
representation. The real subspace $x\in\RR$, $z=\overline y$
is preserved, and $x^2+yz=x^2+|y|^2$ is positive definite there.
These are proved statements about the explicit matrices. Their
relations alone do not prove Conjecture~\ref{conj:dual-fricke}, or
covariance of $V$ under every element of $\Gamma_0(31)^+$.

\subsection{A quadratic link to the Monster series}
The invariant quadratic expression suggests a direct relation with
the classical Hauptmodul already identified in
Section~\ref{sec:target}.

\begin{conjecture}[Quadratic moonshine identity]\label{conj:dual-quadratic}
For the dual components \eqref{eq:dual-splitting},
\begin{equation}\label{eq:dual-quadratic}
 S_0(\tau)^2+S_2(\tau)S_3(\tau)+4
   =\bigl(T_{31A}(\tau)+1\bigr)^2.
\end{equation}
\end{conjecture}

Here $T_{31A}=q^{-1}U(1,31)^3$ is the known classical function,
independently of Conjecture~\ref{conj:main}. Exact integer coefficient
calculations verify \eqref{eq:dual-quadratic} through $q^{519}$,
inclusive. The dual coefficients were computed through $q^{520}$
and compared with an independent Rogers--Ramanujan expansion of
$T_{31A}$. The lost order accounts for the leading term $q^{-1}$
when squaring. This finite check is evidence for the infinite identity.

If Conjecture~\ref{conj:dual-fricke} holds, the quadratic expression
$S_0^2+S_2S_3$ is Fricke-invariant by
Proposition~\ref{prop:dual-matrices}; it is already invariant under
translation. This is compatible with \eqref{eq:dual-quadratic}, but
does not prove that identity. Likewise, neither dual conjecture alone
proves the original positive evaluation: a further identity connecting
the original combination \eqref{eq:signed} to the dual sums would
still be needed. There is no conflict with
Corollary~\ref{cor:fricke}, which concerns individual sums for the
original quadratic datum, rather than this dual combination.

\section{Consequences, scope, and the remaining identities}
The results separate three kinds of obstruction. The positive-support
test excludes some nonintegral quadratic data from any finite positive
integer-power realization. The radial calculation excludes every
individual member of the surviving index-$31$ family. Neither excludes
the explicit positive combination \eqref{eq:positive}.
This distinction explains why the earlier failure of individual searches
does not settle the index-$31$ case.

Theorems about all modular Nahm sums cannot be inferred from the finite
cubic table. Nor does the present work identify all positive combinations
for its surviving rows. At index $13$ the earlier targets are reciprocal
coset products with a two-component Fricke law; at index $31$ the
present target is a theta difference divided by eta factors, with a
scalar Fricke law and cube $T_{31A}$. Their sum evaluations have
different numerators and both remain conjectural.

The concrete remaining task is the mixed-base realization
$q^{-1}\Ssum^3=T_{31A}$, or equivalently the transformation
\eqref{eq:positive-numerator} to \eqref{eq:U31}, equivalently
\eqref{eq:theta-conjecture}. A suitable summation transformation,
finitization, or recurrence certificate would prove it. The comparison
with \eqref{eq:mod5positive} suggests studying identities for entire
families of shifts, but supplies no such certificate at index $31$.
The known theta transformations already establish the modular target;
they do not establish modularity of its proposed sum side.
In particular, no Sturm bound can currently promote
\eqref{eq:finitecheck} to an infinite identity.

The principal unconditional result is the all-rational-shift exclusion
and its Fricke consequence. The positive-support criterion, contiguous
identity, five-term dilogarithm certificate, and amplitude calculation
provide further proved statements.
The specific mixed-base evaluation remains Conjecture~\ref{conj:main}.
For the dual datum, the three-class splitting, translation law, and
icosahedral matrix algebra in Section~\ref{sec:dual-fricke} are also
proved. The proposed Fricke transformation and quadratic identity
with $T_{31A}$ remain Conjectures~\ref{conj:dual-fricke} and
\ref{conj:dual-quadratic}. They offer additional structure for a proof
approach, while leaving the original sum evaluation open.

\section*{Statements and declarations}
\noindent\textbf{Funding.} The author received no external funding for this work.

\smallskip
\noindent\textbf{Competing interests.} The author declares no competing interests.

\smallskip
\noindent\textbf{AI assistance.} ChatGPT (OpenAI) assisted with searches,
symbolic and exact-arithmetic calculations, verification, literature
checks, proof organization, and editing. The author is responsible for
the final manuscript and its mathematical content.

\appendix
\section{An explicit certificate for the leading action}\label{app:dilog}
This appendix proves the value used in \eqref{eq:dilog-value} and
Proposition~\ref{prop:amplitude}. It is the conductor-$19$ certificate
from the companion cubic classification \cite[Appendix~C, identity~8]{Cubic},
reproduced here so that the radial exponent has a self-contained proof.

\begin{proposition}\label{prop:dilog}
For $\rho,\sigma$ in Lemma~\ref{lem:saddle},
$31L(\rho)+L(\sigma)=4\pi^2$.
\end{proposition}
\begin{proof}
Write $[x]$ for a formal symbol, including $[1]$, and put
\begin{align*}
 \mathrm{FT}(x,y)&=[x]+[y]-[xy]
 -\left[\frac{x(1-y)}{1-xy}\right]
 -\left[\frac{y(1-x)}{1-xy}\right],\\
 \mathrm R(x)&=[x]+[1-x]-[1].
\end{align*}
The Rogers five-term and reflection identities send each of these
arrays to zero under $[x]\mapsto L(x)$, $[1]\mapsto\pi^2/6$,
for $x,y\in(0,1)$; see \cite{Zagier}. In the following certificate set
\[
 s=1-\rho,\qquad \ell=\rho^2-\rho+1,\qquad
 e=1+\rho,\qquad h=\frac{1+s}{s}.
\]
These symbols are local to this appendix. Reduction by
$\rho^3=5\rho^2-2\rho-1$ gives the formal identity
\begingroup\small
\begin{multline}\label{eq:dilog-certificate}
-5\,\mathrm{FT}\big(s,\;s\big) + 4\,\mathrm{FT}\big(e^{-1},\;s\,e\big)\\
- \mathrm{FT}\big(\ell^{-1}\,e^{-1},\;\ell\big) + \mathrm{FT}\big(\ell,\;s\,\ell^{-1}\big)\\
+ 2\,\mathrm{FT}\big(\rho^{-1}\,e\,h^{-1},\;s\,e^{-1}\,h\big) + 4\,\mathrm{FT}\big(\rho^{-1}\,s,\;\rho^{-1}\,s\big)\\
+ \mathrm{FT}\big(s^{-1}\,\ell^{-1}\,h^{-1},\;\ell\big) + 2\,\mathrm{FT}\big(\rho^{-2}\,e\,7^{-1},\;\rho^{2}\big)\\
+ \mathrm{FT}\big(\rho^{-1}\,\ell\,e^{-1},\;s\,\ell^{-1}\,e\big) - 2\,\mathrm{FT}\big(h\,7^{-1},\;s^{2}\big)\\
+ \mathrm{FT}\big(\rho^{-1}\,s^{-1}\,\ell\,h^{-1},\;\rho^{-1}\,s\big) - \mathrm{FT}\big(\rho^{-1}\,e\,h^{-1},\;\rho^{-1}\,\ell\,e^{-1}\big)\\
- \mathrm{FT}\big(h^{-1},\;s^{2}\,\ell^{-1}\,h\big) + 4\,\mathrm{FT}\big(\rho^{-2}\,s\,e,\;h\,7^{-1}\big)\\
+ 4\,\mathrm{FT}\big(\rho^{-1}\,s^{2}\,h,\;e\,7^{-1}\big) - 2\,\mathrm{FT}\big(e^{-1}\,h^{-1},\;s\big)\\
+ \mathrm{FT}\big(\rho^{-1}\,s,\;\rho\,\ell^{-1}\big) + \mathrm{FT}\big(\ell^{-1}\,e^{-1},\;s\,\ell\,h^{-1}\big)\\
- 2\,\mathrm{FT}\big(e^{-1},\;h^{-1}\big) + 2\,\mathrm{FT}\big(\rho^{-2}\,s\,e\,h\,7^{-1},\;s^{-1}\,h^{-1}\big)\\
+ \mathrm{FT}\big(\rho^{-2}\,s^{2},\;\rho^{-1}\,e\,h^{-1}\big) - 2\,\mathrm{FT}\big(\rho^{-1}\,s\,e^{-1}\,h^{-1}\,7,\;h\,7^{-1}\big)\\
+ 2\,\mathrm{FT}\big(\rho^{-1}\,s^{2}\,e\,h\,7^{-1},\;\rho\,e^{-1}\big) + \mathrm{FT}\big(\rho^{-3}\,\ell\,e^{-1},\;\rho^{3}\big)\\
- \mathrm{FT}\big(s\,\ell^{-1}\,e,\;s\,e^{-1}\,h^{-1}\big) + \mathrm{FT}\big(\rho^{-6}\,s,\;\rho^{3}\,s^{-1}\,\ell\,e^{-1}\big)\\
+ \mathrm{R}\big(\rho^{-3}\,s^{2}\,e\,h^{-1}\big) + \mathrm{R}\big(\rho^{-3}\,s^{3}\big)\\
- 2\,\mathrm{R}\big(\rho^{-2}\,e^{2}\,7^{-1}\big) - 4\,\mathrm{R}\big(\rho^{-2}\,s\,e\big)\\
+ \mathrm{R}\big(\rho^{-2}\,s^{2}\big) - \mathrm{R}\big(\rho^{-1}\,s^{-1}\,\ell\,h^{-1}\big)\\
- \mathrm{R}\big(s^{-1}\,\ell^{-1}\,h^{-1}\big) - \mathrm{R}\big(s^{-1}\,h^{-1}\big)\\
+ 3\,\mathrm{R}\big(e^{-1}\big) + 30\,\mathrm{R}\big(s\big)\\
- 3\,\mathrm{R}\big(s\,e\big) + \mathrm{R}\big(s^{2}\,\ell^{-1}\,h\big)\\
- \mathrm{R}\big(\rho^{2}\,s^{2}\,e^{-2}\,h^{-1}\,7\big)\\
=31[\rho]+[\sigma]-24[1].
\end{multline}
\endgroup
For completeness, the identity is checked by replacing each
$\mathrm{FT}$ and $\mathrm R$ by its displayed array, reducing each
argument in the basis $1,\rho,\rho^2$, and collecting equal symbols.
The only nonzero coefficients are $31$ at $[\rho]$, $1$ at
$[\rho^{-6}s]=[\sigma]$, and $-24$ at $[1]$; the last equals
minus the sum of the thirteen reflection coefficients.
All inputs lie in $(0,1)$: substitute the rational bounds
$7781238/10^7<\rho<7781239/10^7$ from Lemma~\ref{lem:saddle}
into the positive monomials in $\rho,s,\ell,e,h,7$.
For $\ell$ one may use
$r_-^2-r_++1<\ell<r_+^2-r_-+1$, where $r_-,r_+$ are these endpoints.
The resulting rational interval bounds stay strictly between $0$ and
$1$ for every input; the derived five-term arguments then do so too.
Applying $L$ to \eqref{eq:dilog-certificate} proves the proposition.
\end{proof}

\end{document}